\documentclass[11pt,reqno]{article}
\usepackage[T1]{fontenc}
\usepackage{amsmath,amssymb,amsthm}
\usepackage[margin=1.25in]{geometry}
\usepackage[colorlinks=true,linkcolor=blue,citecolor=blue]{hyperref}

\theoremstyle{plain}
\newtheorem{theorem}{Theorem}
\newtheorem{lemma}[theorem]{Lemma}
\newtheorem{corollary}[theorem]{Corollary}

\theoremstyle{definition}
\newtheorem{remark}[theorem]{Remark}

\newcommand{\R}{\mathbb{R}}
\newcommand{\N}{\mathbb{N}}
\newcommand{\Om}{\Omega}
\newcommand{\Ht}{\widetilde H}
\newcommand{\dsp}{\mathrm{DSp}}

\newcommand{\dr}{\mathrm{DR}}

\newcommand{\dd}{\mathrm{Dom}}

\newcommand{\ip}[2]{(#1,#2)}

\begin{document}

\title{Truncations for fractional Laplacians}
\author{Egor Ignatev\footnote{St.Petersburg State University, 
Universitetskii pr. 28, St.Petersburg, 198504, Russia},~ 
\setcounter{footnote}{6}
Alexander I. Nazarov\footnote{
St.Petersburg Dept. of Steklov Institute, Fontanka 27, St.Petersburg, 191023, 
Russia. Supported by the Ministry of Science and Higher Education of the Russian 
Federation (agreement 075-15-2025-344 dated 29/04/2025 for Saint Petersburg 
Leonhard Euler International Mathematical Institute)}, \\Pavel Nichitenko$^*$ 
and Artur Tursunbaev$^*$
}

\date{}

\maketitle

\begin{abstract}
Let $\Om\subset\R^n$ be a bounded Lipschitz domain. We prove and widely 
generalize a conjecture of A.\,I.~Nazarov \cite{Naz21}: for $s\in(1,\frac 32)$ 
the quadratic form $Q^{\dsp}_s[u]$ of the spectral fractional Dirichlet 
Laplacian
strictly increases under the map $u\mapsto|u|$ provided $u\in\Ht^s(\Om)$ 
changes sign in $\Omega$.
\end{abstract}

\section{Introduction and main results}

Let $\Om\subset\R^n$ be a bounded strictly Lipschitz domain. Denote by
$\lambda_j$, $j\in\N$, the nondecreasing sequence of eigenvalues of the 
Dirichlet Laplacian in $\Om$ counting with their multiplicities, and by 
$\varphi_j$ corresponding eigenfunctions forming a complete orthonormal system 
in $L_2(\Om)$.

The \textit{spectral fractional Dirichlet Laplacian} (sometimes called the 
Navier fractional Laplacian) of order $s>0$ is the positive self-adjoint 
operator with quadratic form
\begin{equation}\label{eq:forms}
 Q^{\dsp}_s[v]=\sum_{j\ge1}\lambda_j^{\,s}\ip{v}{\varphi_j}_{L_2}^2,
 \qquad
\dd(Q^{\dsp}_s)=\{v\in L_2(\Om):Q^{\dsp}_s[v]<\infty\}.
\end{equation}

We recall the definitions of the classical Sobolev--Slobodetskii spaces
in $\mathbb{R}^{n}$ (see \cite[Subsection 2.3.3]{Tr78} or \cite{DNPV}),
\begin{equation*}
H^{s}(\mathbb{R}^{n})=\bigl\{u\in {
\mathcal{S}}'(\mathbb{R}^{n}):\bigl(1+|
\xi |^{2}\bigr)^{
\frac{s}{2}} {\mathcal{F}}u(\xi )\in
L_{2}(\mathbb{R}^{n})\bigr\},
\end{equation*}
and corresponding spaces in $\Omega $ (see \cite[Subsection 4.2.1]{Tr78} 
and \cite[Subsection 
4.3.2]{Tr78}),
\begin{equation*}
H^{s}(\Omega )=\bigl\{u |_{\Omega }:u\in
H^{s}(\mathbb{R}^{n})\bigr\}; \quad
\widetilde{H}^{s}(\Omega )=\bigl\{u\in H^{s}( \mathbb{R}^{n}): 
\operatorname{supp} 
(u)\subset \overline{\Omega
}\bigr\}.
\end{equation*}

We also recall that another fractional Dirichlet Laplacian, namely, the 
\textit{restricted} one is defined by its quadratic form
\begin{equation}\label{Qdr}
Q^{\dr}_s[v]=\int\limits_{\R^n}|\xi|^{2s}|(\mathcal Fv)(\xi)|^2d\xi, \qquad 
\dd(Q^{\dr}_s)=\Ht^s(\Om).
\end{equation}
Notice that (see \cite[Lemma~1]{MN14} and 
\cite[Lemma~2]{MN16})
\begin{equation*}
 \widetilde{H}^{s}(\Omega )=\operatorname{Dom}
(Q_{s}^{\mathrm{DSp}}),\quad s\in(0, \tfrac{3}{2});
\qquad  \widetilde{H}^{s}(\Omega )\subsetneq \operatorname{Dom}
(Q_{s}^{ \mathrm{DSp}}), \quad s\ge \tfrac{3}{2}.
\end{equation*}

Finally, we write $u^\pm=\frac12(|u|\pm u)$. It is easy to see that $u^+u^-=0$, 
and $u$ is sign-changing if and only if simultaneously both $u^+\ne0$ and 
$u^-\ne0$ in $L_2(\Om)$. 

\medskip

It is well known that
if $v\in \widetilde H^1(\Omega)=H^1_0(\Omega)$, then $|v|\in H^1_0(\Omega)$ 
with $\nabla|v|=(\operatorname{sgn}v)\nabla v$
a.e., see \cite[Lemma 7.6]{GT}, so for any $u\in\widetilde{H}^{1}(\Omega )$ we 
have
\begin{equation*}
Q_1^{\mathrm{DR}}[u]=Q_1^{\mathrm{DSp}}[u]=\int\limits_{\Omega} |\nabla u|^2 = 
\int\limits_{\Omega} |\nabla |u||^2 =
 Q_1^{\mathrm{DR}}[|u|]=Q_1^{\mathrm{DSp}}[|u|].
\end{equation*}
In contrast, for $s\in(0,1)$ one has (see \cite[Theorem 3]{MN15})
\begin{equation}\label{eq:thm6}
 Q_s^{\mathrm{DR}}[u]>Q_s^{\mathrm{DR}}[|u|]; 
\qquad Q_s^{\mathrm{DSp}}[u]>Q_s^{\mathrm{DSp}}[|u|]
\end{equation}
for all sign-changing $u\in\widetilde{H}^{s}(\Omega )$.

What happens for $s>1$? If $s\in (1,\frac{3}{2})$ then the operator
$u\mapsto |u|$ is a bounded transform of $H^{s}(\mathbb{R}^{n})$ into itself,
see, e.g., \cite[Section 4]{BS11}.\footnote{To the best of our knowledge, its 
continuity
is still an open problem.} Moreover, it is easy to show that the assumption
$s<\frac{3}{2}$ cannot be improved, see, e.g., \cite[Example~1]{MN19}.

For $s\in (1,\frac{3}{2})$ the restricted form satisfies the opposite 
inequality, see
\cite[Theorem 1]{MN19}:
 for sign-changing $u\in\Ht^s(\Om)$,
\begin{equation}\label{eq:thm7}
 Q^{\dr}_s[u]<Q^{\dr}_s[|u|],\qquad s\in(1,\tfrac32).
\end{equation}
In \cite[Conjecture 8]{Naz21} it was conjectured that inequality similar to 
\eqref{eq:thm7} holds for $Q^{\dsp}_s$.
Here we prove a more general statement.

\begin{theorem}\label{th:main}
Let $s>0$, $s\notin\N$, and let $f,g\in \dd(Q^{\dsp}_s)$ be non-negative 
non-trivial functions such that $fg=0$ a.e. in $\Omega$. Then
$$
\begin{array}{lll}
Q^{\dsp}_s[f+g]<Q^{\dsp}_s[f-g]
& \text{ if ~~$\lfloor s\rfloor$}&\text{ is~ even;}
\\
Q^{\dsp}_s[f+g]>Q^{\dsp}_s[f-g]
& \text{ if ~~$\lfloor s\rfloor$}&\text{ is~ odd}.
\end{array}
$$
If in addition $f,g\in \dd(Q^{\dr}_s)=\Ht^s(\Om)$ then the same inequalities 
hold for $Q^{\dr}_s$.
\end{theorem}

\begin{remark}\label{rem:conj}
 The last statement of Theorem \ref{th:main} was in fact conjectured in 
\cite[Remark 1]{MN19}.
\end{remark}

\begin{corollary}\label{cor:conj}
Let $s\in(1,\frac32)$. If $u\in\Ht^s(\Om)$ is sign-changing, then 
$Q^{\dsp}_s[u]<Q^{\dsp}_s[|u|]$.
\end{corollary}

\begin{remark}\label{rem:LLM}
 The original proof of Corollary \ref{cor:conj} via Lemma \ref{lem:bala} (for 
$m=1$) was given by the LLM Claude (Anthropic), which was directed jointly by 
E.I., P.N., and A.T. 
\end{remark}

The structure of the paper is as follows. In Section 2 we give an 
elementary proof of Theorem \ref{th:main} and Corollary \ref{cor:conj} for 
spectral quadratic form. In Section 3 we 
prove Theorem \ref{th:main} for restricted quadratic form and discuss possible 
generalizations.

\section{The form $Q^{\dsp}_s$}

Since the spectrum of the Dirichlet Laplacian in $\Omega$ is discrete, this 
case is more elementary. For the brevity we denote
$f_j=(f,\varphi_j)_{L_2}$. First we prove an important auxiliary statement.

\begin{lemma}\label{lem:bala}
Let $s>0$, $s\notin\N$. Denote $m=\lfloor s\rfloor$ and define the function
$$
R_m(t)=e^{-t}-\sum\limits_{k=0}^m \frac{(-1)^kt^k}{k!}.
$$ 
Then $(-1)^m R_m(t)<0$ on $(0,+\infty)$, and
\begin{equation}\label{eq:bala}
\int\limits_0^\infty 
R(t\lambda)\,\frac{dt}{t^{1+s}}= \Gamma(-s)\lambda^{s},
 \qquad \lambda\ge0 .
\end{equation}
\end{lemma}

\begin{proof}
The first statement is evident for $m=0$ and follows by induction from the 
relations
$$
R_k(0)=0,\qquad R_k'(t)=-R_{k-1}(t).
$$
Next, it suffices to prove \eqref{eq:bala} for $\lambda=1$. The integrand in 
\eqref{eq:bala} is $O(t^{m-s})$ at zero and $O(t^{m-s-1})$ at infinity, so we 
can integrate by parts:
\begin{multline*}
 \int\limits_0^\infty R_m(t)\,\frac{dt}{t^{1+s}} =\frac 
{-1}s\int\limits_0^\infty R_{m-1}(t)\,\frac{dt}{t^{1+(s-1)}}
 =\dots \\
 =\frac {(-1)^{m+1}}{s(s-1)\dots(s-m)}\int\limits_0^\infty 
\frac{e^{-t}\,dt}{t^{1+(s-m)}}
=\frac 
{(-1)^{m+1}\Gamma(m+1-s)}{s(s-1)\dots(s-m)}=\Gamma(-s).
\end{multline*}
\end{proof}

\begin{lemma}\label{lem:pos}
Let $f,g\in L_2(\Om)$, $f,g\ge0$, $f\ne0$, $g\ne0$. Then for every $t>0$ we have
$$
\sum_{j\ge1}e^{-t\lambda_j}f_jg_j>0.
$$
\end{lemma}

\begin{proof}
It is evident that
$$
\sum_{j\ge1}e^{-t\lambda_j}f_jg_j=\int\limits_{\Om}V(x,t)g(x)\,dx,\qquad 
V(x,t)=\sum_{j\ge1}e^{-t\lambda_j}f_j\varphi_j(x).
$$
Next, it is well known that $V(x,t)$ is the solution of the initial-boundary 
value problem for the heat equation
$$
\partial_tV-\Delta V=0\quad \mbox{in} \ \ \Om\times (0,+\infty), \qquad 
V|_{x\in\partial\Om}=0, \qquad V|_{t=0}=f(x). 
$$
By the strong maximum principle (see, e.g., \cite[Ch.3, Theorem 5]{PW}), $V>0$ 
in $(0,\infty)\times\Om$, and the statement follows.
\end{proof}

\begin{lemma}\label{lem:vanish}
Let $s>0$, and let $f$ and $g$ satisfy the assumptions of Theorem 
\ref{th:main}. Then for every $k=0,\dots, \lfloor s\rfloor$ we have
$$
\sum_{j\ge1}\lambda_j^kf_jg_j=0.
$$
\end{lemma}

\begin{proof}
Since $f,g\in\dd(Q^{\dsp}_s)$, it is easy to see that for $k=0,\dots, \lfloor 
s\rfloor$
\begin{equation}\label{eq:vanish}
\sum_{j\ge1}\lambda_j^kf_jg_j=
\begin{cases}
 \int\limits_{\Om}(-\Delta)^{\frac k2}f(x)\, (-\Delta)^{\frac 
k2}g(x)\,dx, & \text{ if ~~$k$~ is~ even},\\
 \int\limits_{\Om}\nabla(-\Delta)^{\frac {k-1\vphantom{1^1}}2}f(x)\, 
\nabla(-\Delta)^{\frac {k-1}2}g(x)\,dx, & \text{ if ~~$k$~ is~ 
odd}.
\end{cases}
\end{equation}
By \cite[Lemma 7.7]{GT}, $\nabla f=0$ a.e. on the set where $f=0$. Since 
$fg=0$ a.e. in $\Om$, scalar products in the right-hand side of 
\eqref{eq:vanish} vanish. 
 \end{proof}

\begin{proof}[Proof of Theorem \ref{th:main} for the form $Q^{\dsp}_s$]
It is easy to see that  
\begin{equation}\label{eq:diff}
Q^{\dsp}_s[f+g]-Q^{\dsp}_s[f-g]=4\sum_{j\ge1}\lambda_j^sf_jg_j.
\end{equation}
We use subsequently Lemma \ref{lem:bala}, Fubini's theorem and Lemma 
\ref{lem:vanish} to get
\begin{multline*}
 \sum_{j\ge1}\lambda_j^sf_jg_j=\frac 
1{\Gamma(-s)}\sum_{j\ge1}\int\limits_0^\infty 
R(t\lambda_j)\,\frac{dt}{t^{1+s}}\,f_jg_j
\\
=\frac 1{\Gamma(-s)}\int\limits_0
^\infty \sum_{j\ge1}
R(t\lambda_j)f_jg_j\,\frac{dt}{t^{1+s}}
=\frac 1{\Gamma(-s)}\int\limits_0
^\infty \sum_{j\ge1}
e^{-t\lambda_j}f_jg_j\,\frac{dt}{t^{1+s}}.
\end{multline*}
By Lemma \ref{lem:pos}, the last integral is positive. Thus, the sign of 
the left-hand side in \eqref{eq:diff} coincides with the sign of $\Gamma(-s)$, 
and the statement follows.
\end{proof}

\begin{proof}[Proof of Corollary \ref{cor:conj}]
As was mentioned in the Introduction, the map $v\mapsto|v|$ takes $H^s(\R^n)$ 
into itself for
$s\in(1,\frac32)$. Therefore, $u^\pm\in\Ht^s(\Om)$ for $u\in\Ht^s(\Om)$.

Next, we recall that $\dd(Q^{\dsp}_s)=\Ht^s(\Om)$ for
$0<s<\frac32$ on a bounded Lipschitz domain. So, Theorem \ref{th:main} is 
applicable with $f=u^+$, $g=u^-$.
\end{proof}

\section{The form $Q^{\dr}_s$ and generalizations}

Notice that the form $Q^{\dr}_s$ is simply quadratic form of the conventional 
fractional Laplacian $-\Delta_{\R^n}^s$ restricted to $\Ht^s(\Om)$. Since the 
spectrum of the Laplacian in $\R^n$ is purely continuous, we recall some facts 
from the theory of spectral measure (see \cite[Ch.~5]{BS87}).

The spectral measure of 
$-\Delta_{\R^n}$ is the projector-valued measure $E$ on $\R$ supported on 
$[0,\infty)$,
such that
$$
-\Delta_{\R^n}=\int\limits_{[0,\infty)} \lambda ~\!dE(\lambda).
$$
Let $f,g\in L_2(\R^n)$. For a Borel set $B$,  set 
$$
\rho_{f,g}(B):=\ip{E(B)f}{g}_{L_2}.
$$
Then $\rho$ is a finite signed Borel measure on $[0,\infty)$, bilinear and 
symmetric in $f,g$, and
$$
\rho_{f,g}([0,\infty))=\ip{f}{g}_{L_2}\quad \mbox{and}\quad
\int\limits_{[0,\infty)}\lambda^s\,d\rho_{f,f}(\lambda)=Q_s^{\dr}[f]\quad 
\mbox{for}\quad f\in H^s(\R^n).
$$ 

\begin{lemma}\label{lem:pos1}
Let $f,g\in L_2(\R^n)$, $f,g\ge0$, $f\ne0$, $g\ne0$. Then for every $t>0$ we 
have
$$
\int\limits_{[0,\infty)}e^{-t\lambda}\,d\rho_{f,g}(\lambda)>0.
$$
\end{lemma}

\begin{proof}
Similarly to Lemma \ref{lem:pos},
$$
\int\limits_{[0,\infty)}e^{-t\lambda}\,d\rho_{f,g}(\lambda)=\int\limits_{\R^n}
V(x, t)g(x)\,dx,
\qquad 
V(\cdot,t)=\int\limits_{[0,\infty)}e^{-t\lambda}\,dE(\lambda) f(\cdot).
$$
It is well known that $V(x,t)$ is the solution of the Cauchy problem 
for the heat equation
$$
\partial_tV-\Delta V=0\quad \mbox{in} \ \ \R^n\times (0,+\infty), \qquad 
V|_{t=0}=f(x). 
$$
Thus, $V$ is the convolution of $f$ with the Gaussian kernel. So, $V>0$ 
in $(0,\infty)\times\R^n$, and the statement follows.
\end{proof}

\begin{proof}[Proof of Theorem \ref{th:main} for the form $Q^{\dr}_s$]
In this case 
\begin{equation*}
Q^{\dr}_s[f+g]-Q^{\dr}_s[f-g]=4\int\limits_{[0,\infty)}\lambda^s\,d\rho_{f,g}(
\lambda).
\end{equation*}
By Lemma \ref{lem:bala} and Fubini's theorem we have
$$
 \int\limits_{[0,\infty)}\lambda^s\,d\rho_{f,g}(
\lambda)=\frac 
1{\Gamma(-s)}\int\limits_{[0,\infty)}\int\limits_0
^\infty
R(t\lambda)\,\frac{dt}{t^{1+s}}\,d\rho_{f,g}(
\lambda)
=\frac 1{\Gamma(-s)}\int\limits_0
^\infty \int\limits_{[0,\infty)}
R(t\lambda)\,d\rho_{f,g}(
\lambda)\,\frac{dt}{t^{1+s}}.
$$
Just as Lemma \ref{lem:vanish}, the power terms in the inner integral vanish, 
and we obtain
\begin{equation}\label{eq:exp}
 \int\limits_{[0,\infty)}\lambda^s\,d\rho_{f,g}(
\lambda)=\frac 1{\Gamma(-s)}\int\limits_0
^\infty 
\int\limits_{[0,\infty)}e^{-t\lambda}\,d\rho_{f,g}(\lambda)\,\frac{dt}{t^{1+s}}.
\end{equation}
By Lemma \ref{lem:pos1}, the last integral is positive,
and the statement follows.
\end{proof}

\begin{remark}\label{rem:known}
For $s\in(1,\tfrac 32)$ and $u\in H^s(\R^n)$ sign-changing, we can put 
$f=u^+$, $g=u^-$ and obtain an alternative proof of \cite[Theorem 1]{MN19}.
\end{remark}

\begin{remark}
Let us mention briefly an analog of Theorem \ref{th:main} for negative $s$. In 
this case the corresponding statement reads as follows:
\medskip

\noindent\textit{Let $s=-\sigma<0$, and let $f,g\in L_2(\Omega)$ be 
non-negative 
non-trivial functions. Then\footnote{For the restricted quadratic form, the same 
is true under additional 
assumption $s>-\frac n2$ ensuring convergence of the integral in \eqref{Qdr}.}}
$$
Q^{\dsp}_{s}[f+g]>Q^{\dsp}_{s}[f-g].
$$

Since for $s<0$ the equality \eqref{eq:exp} follows immediately from the 
definition of Gamma-function, the proof runs without assumption $fg=0$.

However, a direct proof can be obtained from the strong maximum principles for 
conventional Dirichlet Laplacian and for spectral fractional Dirichlet 
Laplacian (see 
\cite{CDDS}). Namely, we have
$$
Q^{\dsp}_{s}[f+g]-Q^{\dsp}_{s}[f-g]=4\int\limits_{\Om}\int\limits_{\Om} 
\mathcal G^{\dsp}_\sigma(x,y) f(y) g(x)\,dydx,
$$
where $\mathcal G^{\dsp}_\sigma$ is the Green function for the spectral 
fractional 
Dirichlet Laplacian of order $\sigma$.

It is easy to see that $\mathcal G^{\dsp}_\sigma$ is the convolution of  
$\mathcal G^{\dsp}_{\sigma-m}$ and $m$ Green functions $\mathcal G$ for 
standard Dirichlet Laplacian (here $m=\lfloor \sigma\rfloor$). It is well known 
that $\mathcal G>0$ in $\Om\times\Om$, while $\mathcal G^{\dsp}_{\sigma-m}>0$ 
in $\Om\times\Om$ by \cite[Lemma 2.6]{CDDS}. So, $\mathcal 
G^{\dsp}_{\sigma}>0$ in $\Om\times\Om$, and we are done\footnote{Notice that 
the order $-\sigma$ restricted fractional Laplacian, in contrast to the 
spectral one, is not inverse to the corresponding operator of order $\sigma$. 
So, the proof via the Green function does not work in this case.}.
\end{remark}

\begin{remark}\label{rem:DR}
It is easy to see that the proof of Theorem \ref{th:main} is insensitive to the 
type of the spectral measure. So, similar statements hold true for the 
quadratic forms of spectral powers for all non-negative self-adjoint 
differential operators such that corresponding parabolic operators obey the 
strong maximum principle, for instance:
\begin{itemize}
\item Neumann and Robin Laplacians in bounded smooth domains;
\item Dirichlet Laplacian in unbounded domains;
\item self-adjoint second order elliptic operators with smooth coefficients, 
etc.  
\end{itemize}
\end{remark}

\end{document}